\documentclass[12pt,british]{amsart}
\usepackage{babel}
\usepackage[latin9]{inputenc}
\usepackage{geometry}
\usepackage{units}
\usepackage{mathtools}
\usepackage{amsmath}
\usepackage{amsthm}
\usepackage{amssymb}
\usepackage{stmaryrd}
\usepackage{graphicx}
\usepackage[unicode=true,pdfusetitle,
 bookmarks=true,bookmarksnumbered=false,bookmarksopen=false,
 breaklinks=false,pdfborder={0 0 1},backref=false,colorlinks=false]
 {hyperref}

\makeatletter

\numberwithin{equation}{section}
\numberwithin{figure}{section}

\theoremstyle{plain}
\newtheorem{thm}{Theorem}[section]
\newtheorem{cor}[thm]{Corollary}
\newtheorem{lem}[thm]{Lemma}
\newtheorem{prop}[thm]{Proposition}

\theoremstyle{definition}
\newtheorem{defn}[thm]{Definition}

\theoremstyle{remark}
\newtheorem{rem}[thm]{Remark}

\makeatother

\DeclarePairedDelimiter{\paren}{\lparen}{\rparen}%
\DeclarePairedDelimiter{\pbrace}{\lbrace}{\rbrace}%
\DeclarePairedDelimiter{\abs}{\lvert}{\rvert}%
\DeclarePairedDelimiter{\norm}{\lVert}{\rVert}%

\newcommand{\Aut}{\operatorname{Aut}}
\newcommand{\diag}{\operatorname{diag}}
\newcommand{\im}{\operatorname{im}}
\renewcommand{\Re}{\operatorname{Re}}

\begin{document}

\title[Punctured parabolic cylinders]{Punctured parabolic cylinders in automorphisms of $\mathbb{C}^{2}$}
\author[J. Reppekus]{Josias Reppekus}
\thanks{The author acknowledges the MIUR Excellence Department Project awarded
to the Department of Mathematics, University of Rome Tor Vergata,
CUP E83C18000100006}
\address{Dipartimento di Matematica\\
Universit\`{a} di Roma ``Tor Vergata''\\
Via Della Ricerca Scientifica 1 \\
00133, Roma, Italy}
\email{reppekus@mat.uniroma2.it}


\begin{abstract}
We show the existence of automorphisms $F$ of $\mathbb{C}^{2}$ with
a non-recurrent Fatou component $\Omega$ biholomorphic to $\mathbb{C}\times\mathbb{C}^{*}$
that is the basin of attraction to an invariant entire curve on which
$F$ acts as an irrational rotation. We further show that the biholomorphism
$\Omega\to\mathbb{C}\times\mathbb{C}^{*}$ can be chosen such that
it conjugates $F$ to a translation $(z,w)\mapsto(z+1,w)$, making
$\Omega$ a parabolic cylinder as recently defined by L.~Boc Thaler,
F.~Bracci and H.~Peters. $F$ and $\Omega$ are obtained by blowing
up a fixed point of an automorphism of $\mathbb{C}^{2}$ with a Fatou
component of the same biholomorphic type attracted to that fixed point,
established by F.~Bracci, J.~Raissy and B.~Stens{\o}nes. A crucial
step is the application of the density property of a suitable Lie
algebra to show that the automorphism in their work can be chosen
such that it fixes a coordinate axis. We can then remove the proper
transform of that axis from the blow-up to obtain an $F$-stable subset
of the blow-up that is biholomorphic to $\mathbb{C}^{2}$. Thus we
can interpret $F$ as an automorphism of $\mathbb{C}^{2}$.
\end{abstract}

\maketitle
\newpage{}

\section{Introduction}

When studying the behaviour of iterates of a holomorphic endomorphism
$F$ of $\mathbb{C}^{d}$, $d\ge1$, one of the basic objects of interest
is the \emph{Fatou set} of all points in $\mathbb{C}^{d}$ that admit
a neighbourhood on which $\{F^{n}\}_{n\in\mathbb{N}}$ is normal.
The connected components of the Fatou set are called the \emph{Fatou
components} of $F$. They can be thought of as maximal connected subsets
of $\mathbb{C}^{2}$ on which the behaviour of $F$ is roughly the
same. A Fatou component $V$ of $F$ is \emph{invariant}, if $F(V)=V$.
It is \emph{recurrent} if it contains an accumulation point of an
orbit $\{F^{n}(p)\}_{n\in\mathbb{N}}$ for some $p\in V$ and \emph{non-recurrent}
or \emph{transient} otherwise. $V$ is \emph{attracting} to a point
$p$ in the closure $\overline{V}$ of $V$ if all orbits starting
in $V$ converge to $p$.

In one variable, an invariant Fatou component $V$ of an entire function
is either attracting to a point in $\overline{V}$, or all orbits
in $V$ escape to $\infty$ ($V$ is a \emph{Baker domain}), or $V$
is a \emph{rotation domain}, i.e.\ there exists a subsequence $\{F^{n_{k}}\}_{k}$
converging to the identity on $V$. In other words, the images of
limit functions of convergent subsequences of $\{F^{n}\}_{n}$ on
$V$ are unique and have dimension $0$ or full dimension $1$ respectively.

Recurrent Fatou components of polynomial automorphisms of $\mathbb{C}^{2}$
have been classified in \cite{BedfordSmillie.1991.Polynomialdiffeomorphismsof(mathbbC2).II:Stablemanifoldsandrecurrence.},
\cite{FornessSibony.1995.Classificationofrecurrentdomainsforsomeholomorphicmaps.}
and \cite{Ueda.2008.HolomorphicmapsonprojectivespacesandcontinuationsofFatoumaps.}.
\cite{ArosioBeniniFornessPeters.2019.DynamicsoftranscendentalHenonmaps.}
generalises these results to automorphisms $F$ of $\mathbb{C}^{2}$
with constant Jacobian and shows that in this case a recurrent Fatou
component $V$ of $F$ is either the basin of an attracting fixed
point in $V$ (biholomorphic to $\mathbb{C}^{2}$ by \cite[Theorem~2]{PetersVivasWold.2008.AttractingbasinsofvolumepreservingautomorphismsofBbbCk}
and the appendix of \cite{RosayRudin.1988.HolomorphicmapsfrombfCntobfCn}),
or a rotation domain, or a \emph{recurrent} \emph{Siegel} or\emph{
Hermann cylinder}, i.e.\ there is a biholomorphism from $V$ to $A\times\mathbb{C}$,
with $A\subseteq\mathbb{C}$ a domain invariant under rotations that
conjugates $F$ to $(z,w)\mapsto(\lambda z,aw)$ with $|\lambda|=1$
and $|a|<1$ on $V$.

By \cite{LyubichPeters.2014.ClassificationofinvariantFatoucomponentsfordissipativeHenonmaps}
every non-recurrent invariant Fatou component of a polynomial automorphism
with sufficiently small Jacobian is attracting to a parabolic-attracting
fixed point in the boundary. Moreover, by \cite{Ueda.1986.Localstructureofanalytictransformationsoftwocomplexvariables.I},
every attracting non-recurrent invariant Fatou component of a polynomial
automorphism is biholomorphic to $\mathbb{C}^{2}$ and admits coordinates
conjugating it to a translation $(z,w)\mapsto(z+1,w)$. Outside the
polynomial setting, the classification of non-recurrent invariant
Fatou components is far from complete and several new phenomena occur:

The first construction of automorphisms of $\mathbb{C}^{d}$ with
an attracting Fatou component that is not simply connected appeared
in \cite{StensonesVivas2014BasinsofAttractionofAutomorphismsinC3}
for $d\ge3$. In \cite{BracciRaissyStensones.2017.AutomorphismsofmathbbC2withaninvariantnon-recurrentattractingFatoucomponentbiholomorphictomathbbCtimesmathbbC}
the authors construct automorphisms of $\mathbb{C}^{2}$ with an attracting
non-recurrent invariant Fatou component biholomorphic to $\mathbb{C}\times\mathbb{C}^{*}$
(see also \cite{Reppekus.2019.PeriodiccyclesofattractingFatoucomponentsoftypemathbbCtimesmathbbCd-1inautomorphismsofmathbbCd}
for multiple such components). We show in Proposition~\ref{prop:BRSniceCoords}
that on these Fatou components the automorphisms are again conjugated
to a translation $(z,w)\mapsto(z+1,w)$. It is an open question whether
these are the only possible biholomorphic types of non-recurrent attracting
invariant Fatou components of automorphisms of $\mathbb{C}^{2}$ and
if they all admit such a conjugation. It is not even clear that these
are the only homotopy types that can occur.

In \cite{JupiterLilov.2004.InvariantnonrecurrentFatoucomponentsofautomorphismsof(mathbbC2).}
the authors take first steps towards narrowing down the possible invariant
non-recurrent Fatou components of automorphisms of $\mathbb{C}^{2}$.
They split their discussion according to the rank of limit maps of
$\{F^{n}\}_{n}$ on the Fatou component $V$. In case all limit maps
have rank $0$, they show that $V$ is either attracting or the images
of the limit maps form an uncountable set without isolated points
contained in a subvariety of fixed points. The eigenvalues in each
of these points are $\{1,\alpha\}$ where $\alpha$ is a unique non-diophantine
rotation. There are no known examples with more than one rank $0$
limit map.

In the case of rank $1$ limit maps \cite{BocBracciPeters.2019.AutomorphismsofmathbbC2withnon-recurrentSiegelcylinders}
defines and gives examples of parabolic cylinders biholomorphic to
$\mathbb{C}^{2}$ (called non-recurrent Siegel cylinders in an earlier
version of \cite{BocBracciPeters.2019.AutomorphismsofmathbbC2withnon-recurrentSiegelcylinders})
in the following sense:
\begin{defn}
Let $F$ be a self-map of $\mathbb{C}^{2}$. The \emph{$\omega$-limit
set} $\omega_{F}(p)$ of a point $p\in\mathbb{C}^{2}$ or $\omega_{F}(U)$
of an open set $U\subseteq\mathbb{C}^{2}$ under $F$ is the set of
all accumulation points of orbits under $F$ starting in $p$ or $U$
respectively.
\end{defn}

\begin{rem}
A Fatou component $U$ of $F$ is non-recurrent (or transient), if
and only if $\omega_{F}(U)\cap U=\emptyset$.
\end{rem}

\begin{defn}
An invariant non-recurrent Fatou component $V$ of $F$ is called
a \emph{parabolic cylinder}, if
\begin{enumerate}
\item the closure of $\omega_{F}(V)$ contains an isolated fixed point,
\item there is an injective holomorphic map $\Phi:V\to\mathbb{C}^{2}$ conjugating
$F$ to the translation $(z,w)\mapsto(z+1,w)$,
\item all limit maps of $\{F^{n}\}_{n}$ on $\Omega$ have dimension $1$.
\end{enumerate}
\end{defn}

\cite{JupiterLilov.2004.InvariantnonrecurrentFatoucomponentsofautomorphismsof(mathbbC2).}
gives examples of Fatou components with a unique rank $1$ limit map
and with an uncountable family of rank $1$ limit maps with identical
images. The latter are a subclass of the parabolic cylinders examined
in \cite{BocBracciPeters.2019.AutomorphismsofmathbbC2withnon-recurrentSiegelcylinders}.
The authors further show that the images of two limit maps of rank
$1$ are either disjoint or intersect in a relatively open subset.
There are no known examples of rank $1$ limit maps with non-identical
images or limit maps of different rank.

In this paper we show the following:
\begin{thm}
\label{thm:PuncturedSiegelCyl}There exist automorphisms $F$ of $\mathbb{C}^{2}$
with a parabolic cylinder $\Omega$ biholomorphic to $\mathbb{C}\times\mathbb{C}^{*}$
and an invariant entire curve $\mathcal{C}=\mathbb{C}\times\{0\}$
in the boundary of $\Omega$ on which $F$ acts as an irrational rotation
around $(0,0)$ such that
\begin{enumerate}
\item the stable set $W^{s}(\mathcal{C}):=\{p\in\mathbb{C}^{2}\mid\omega_{F}(p)\subseteq\mathcal{C}\}$
of $\mathcal{C}$ is $\Omega\cup\mathcal{C}$, i.e.\ $\Omega$ contains
precisely all orbits approaching $\mathcal{C}$ non-trivially.
\item the $\omega$-limit set $\omega_{F}(\Omega)$ of $\Omega$ is $\mathcal{C}^{*}:=\mathbb{C}^{*}\times\{0\}$
and the limit maps of $\{F^{n}\}_{n}$ on $\Omega$ all have image
$\mathcal{C}^{*}$ and differ precisely by postcomposition with arbitrary
rotations of \emph{$\mathcal{C}^{*}$}.
\end{enumerate}
\end{thm}

The parabolic cylinder in the above theorem is \emph{punctured} in
that it is biholomorphic to $\mathbb{C}\times\mathbb{C}^{*}$ and
has as its $\omega$-limit set a punctured Siegel curve $\mathcal{C}^{*}$,
i.e.\ an entire curve $\mathcal{C}$ on which $F$ is conjugated
to an irrational rotation minus the unique fixed point of $F$ in
$\mathcal{C}$.

\begin{figure}
$\begin{array}{c}\includegraphics[width=0.5\columnwidth]{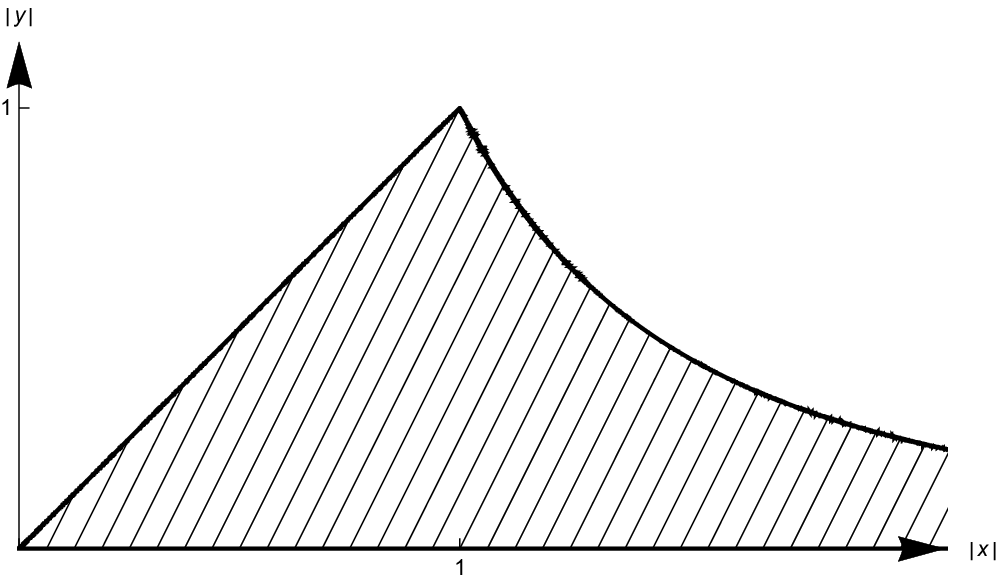}\end{array}\bigtimes\quad{}\begin{array}{c}\includegraphics[bb=0bp 0bp 420bp 293bp,width=0.4\columnwidth]{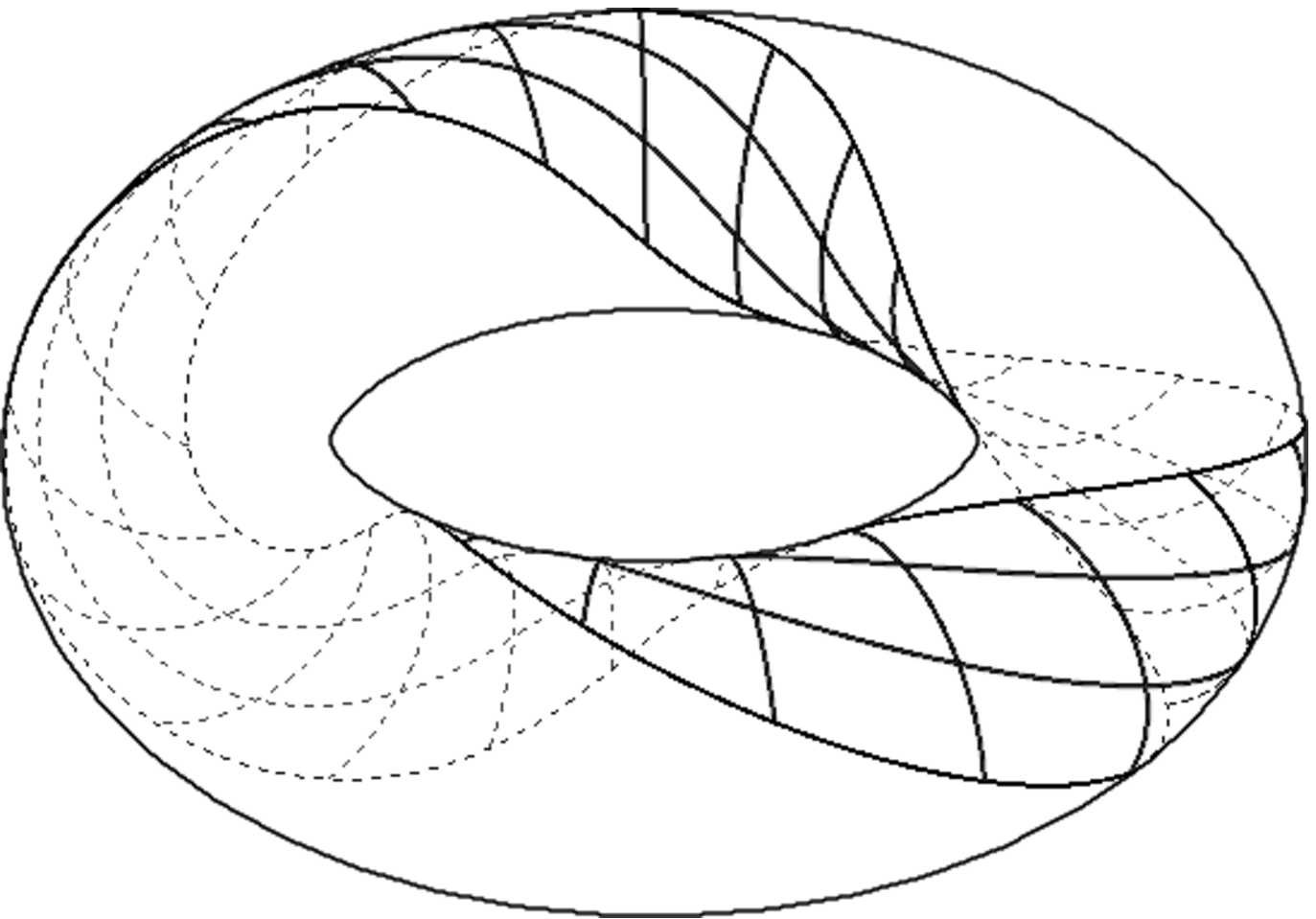}\\ (\arg x,\arg y) \in (S^1)^2 \end{array}$

\caption{\label{fig:Argument-components}Decomposition of $B$ in modulus and
argument components}
\end{figure}
The automorphisms in Theorem~\ref{thm:PuncturedSiegelCyl} have,
near $\mathbb{C}\times\{0\}$, the form
\begin{equation}
F(x,y)=(\lambda^{2}x+R(x,y),\overline{\lambda}y(1-xy^{2}/2)+yO(\norm{(xy,y)}^{l})),\label{eq:AutomAbove}
\end{equation}
where $R(x,y)=xO(\|(xy,y)\|^{l})+O(\|(xy,y)\|^{l})$ and $\lambda\in S^{1}$
is a Brjuno number. They are obtained by lifting to the blow-up at
the origin automorphisms of the form
\begin{equation}
\check{F}(z,w)=\paren{\lambda z,\overline{\lambda}w}\cdot\paren [\Big]{1-\frac{zw}{2}}+wO(\norm{(z,w)}^{l})\label{eq:AutomBelow}
\end{equation}
with $l\in\mathbb{N}_{0}$ sufficiently large, for which \cite{BracciRaissyStensones.2017.AutomorphismsofmathbbC2withaninvariantnon-recurrentattractingFatoucomponentbiholomorphictomathbbCtimesmathbbC}
established the existence of a non-recurrent Fatou component $\check{\Omega}$
attracted to the origin and biholomorphic to $\mathbb{C}\times\mathbb{C}^{*}$.
The parabolic cylinder $\Omega$ is the proper transform of $\check{\Omega}$
and contains an $F$-invariant subset $B$ eventually containing any
orbit in $\Omega$ given by
\[
B=\{(x,y)\in\mathbb{C}^{2}\mid xy^{2}\in S,|x|<\min\{|y|^{2\gamma},|y|^{\gamma-1}\}\},
\]
where $S$ is a small sector with vertex at $0$ around the positive
real axis in $\mathbb{C}$ and $\gamma\in(0,1)$. Figure~\ref{fig:Argument-components}
shows $B$ as a product in polar decomposition (barring some truncation
away from the limit set $\mathbb{C}^{*}\times\{0\}$ depending on
$S$).

\subsection*{Outline}

In Section \ref{sec:The-automorphism}, we construct our family of
automorphisms. We first use results from \cite{Varolin.2001.Thedensitypropertyforcomplexmanifoldsandgeometricstructures.}
and \cite{Varolin.2000.Thedensitypropertyforcomplexmanifoldsandgeometricstructures.II.}
to show the existence of automorphisms $\check{F}$ of the form (\ref{eq:AutomBelow}).
We then blow up at the origin and observe that the lift $F$ of $\check{F}$
leaves invariant the proper transform of the $z$-axis. Removing this
subvariety from the blow-up leaves us with a copy of $\mathbb{C}^{2}$
on which $F$ acts as an automorphism.

In Section~\ref{sec:The-Fatou-component} we use estimates on orbit
behaviour in the Fatou component $\check{\Omega}$ to show that the
proper transform $\Omega$ is still a Fatou component of $F$.

Then we establish coordinates on $\Omega$ conjugating $F$ to $(z,w)\mapsto(z+1,w)$
in Section~\ref{sec:Cylinder-coordinates}, which we use in Section~\ref{sec:Limit-set}
to identify the images of limit maps.

\subsection*{Conventions}

We use the following notations for asymptotic behaviour (as $x\to x_{0}$):
\begin{itemize}
\item $f(x)=O(g(x))$, if $\limsup_{x\to x_{0}}\frac{|f(x)|}{|g(x)|}=C<+\infty$,
\item $f(x)\approx g(x)$, if $f(x)=O(g(x))$ and $g(x)=O(f(x))$,
\item $f(x)\sim g(x)$, if $\lim_{x\to x_{0}}\frac{|f(x)|}{|g(x)|}=1$.
\end{itemize}

\subsection*{Acknowledgements}

The author would like to thank Riccardo Ugolini for the introduction
to D.~Varolin's work, Filippo Bracci for continual advice, and the two
referees for  suggesting a stronger formulation of the main result and
useful comments improving the presentation of the paper.

\section{\label{sec:The-automorphism}The family of automorphisms}

In this section we show that the automorphisms with non-recurrent
Fatou components biholomorphic to $\mathbb{C}\times\mathbb{C}^{*}$
and attracted to the origin constructed in \cite{BracciRaissyStensones.2017.AutomorphismsofmathbbC2withaninvariantnon-recurrentattractingFatoucomponentbiholomorphictomathbbCtimesmathbbC}
can be chosen such that their lift to the blow-up at the origin can
be restricted to an automorphism of a subset biholomorphic to $\mathbb{C}^{2}$.

We first recall the framework of \cite{BracciRaissyStensones.2017.AutomorphismsofmathbbC2withaninvariantnon-recurrentattractingFatoucomponentbiholomorphictomathbbCtimesmathbbC}:
Let $F_{{\rm N}}$ be a germ of biholomorphisms of $\mathbb{C}^{2}$
at the origin given by
\[
F_{{\rm N}}(z,w):=\paren{\lambda z,\overline{\lambda}w}\cdot\paren [\Big]{1-\frac{zw}{2}},
\]
where $\lambda\in S^{1}$ is Brjuno, i.e.
\begin{equation}
-\sum_{\nu=1}^{\infty}2^{-\nu}\log\omega(2^{\nu})<\infty,\label{eq:Brjuno}
\end{equation}
where $\omega(m):=\min\{\abs{\lambda^{k}-\lambda}\mid2\le k\le m\}$
for $m\ge2$.
\begin{defn}
For $r>0$, $\theta\in(0,\pi/2)$, and $\beta\in(0,1/2)$ let
\begin{align*}
W(\beta) & :=\pbrace{(z,w)\in\mathbb{C}^{2}\mid|z|<|zw|^{\beta},|w|<|zw|^{\beta}},\\
S(r,\theta) & :=\{u\in\mathbb{C}\mid|\arg(u)|<\theta,|u-r|<r\},
\end{align*}
and
\[
\check{B}(r,\theta,\beta):=\pbrace{(z,w)\in W(\beta)\mid zw\in S(r,\theta)}.
\]
\end{defn}

The main result in \cite{BracciRaissyStensones.2017.AutomorphismsofmathbbC2withaninvariantnon-recurrentattractingFatoucomponentbiholomorphictomathbbCtimesmathbbC}
(globalising a local result in \cite{BracciZaitsev.2013.Dynamicsofone-resonantbiholomorphisms.})
is:
\begin{thm}
\label{thm:BRS}Let $l\in\mathbb{N}_{\ge4}$, $\theta_{0}\in(0,\pi/2)$,
$\beta_{0}\in(0,1/2)$ such that $\beta_{0}(l+1)\ge4$. Then there
exist automorphisms $\check{F}$ of $\mathbb{C}^{2}$ such that
\begin{equation}
\check{F}(z,w)=F_{{\rm N}}(z,w)+O(\|(z,w)\|^{l})\label{eq:BZgerms}
\end{equation}
near the origin and every automorphism of the form (\ref{eq:BZgerms})
has an non-recurrent invariant Fatou component $\check{\Omega}$ attracted
to $(0,0)$ and biholomorphic to $\mathbb{C}\times\mathbb{C}^{*}$,
that contains a local (uniform) basin of attraction $\check{B}:=\check{B}(r_{0},\beta_{0},\theta_{0})$
for some $r_{0}>0$, that eventually contains any orbit in $\check{\Omega}$,
i.e.\ $\check{F}(\check{B})\subseteq\check{B}$, $\lim_{n\to\infty}\check{F}^{n}\equiv(0,0)$
uniformly in $\check{B}$, and $\check{\Omega}=\bigcup_{n\in\mathbb{N}}\check{F}^{-n}(\check{B})$.
\end{thm}

Next we show the above class of automorphisms contains elements fixing
an axis. D.~Varolin's work on the density property shows in particular:
\begin{thm}
\label{thm:JetInterpFixingAxis}For every invertible germ of automorphisms
$G_{0}$ of $\mathbb{C}^{2}$ at the origin pointwise fixing $\{w=0\}$
and every $l\in\mathbb{N}$, there exists an automorphism $G\in\Aut(\mathbb{C}^{2})$
such that
\begin{equation}
G(z,w)=G_{0}(z,w)+wO\paren{{\norm{z,w}}^{l}}.\label{eq:JetInterpFixAx}
\end{equation}
\end{thm}

\begin{proof}
By \cite[Theorem~5.1]{Varolin.2001.Thedensitypropertyforcomplexmanifoldsandgeometricstructures.},
the Lie algebra $\mathfrak{g}$ of holomorphic vector fields on $\mathbb{C}^{2}$
that vanish on $\mathbb{C}\times\{0\}$ has the \emph{density property},
i.e.\ the complete vector fields are dense in $\mathfrak{g}$. \cite[Theorem~1]{Varolin.2000.Thedensitypropertyforcomplexmanifoldsandgeometricstructures.II.}
states that for such a Lie algebra, if a germ can be interpolated
up to some order $l\in\mathbb{N}$ (i.e. matched up to order $l$)
by compositions of flows of vector fields in $\mathfrak{g}$, the
same can be done using only flows of \emph{complete} vector fields
in $\mathfrak{g}$. Flows of complete vector fields in $\mathfrak{g}$
are automorphisms of $\mathbb{C}^{2}$ fixing $\{w=0\}$. By \cite[Example~1]{Varolin.2000.Thedensitypropertyforcomplexmanifoldsandgeometricstructures.II.}
, the germs that can be interpolated in this way (to arbitrary order
$l\in\mathbb{N}$) are precisely the ones fixing $\{w=0\}$ pointwise.
\end{proof}
Let $\Lambda=\diag(\lambda,\overline{\lambda})$. Applying Theorem~\ref{thm:JetInterpFixingAxis}
to $G_{0}=\Lambda^{-1}F_{{\rm N}}$, we obtain an automorphism $\check{F}=\Lambda G\in\Aut(\mathbb{C}^{2})$
fixing $L:=\{w=0\}$ as a set and interpolating $F_{{\rm N}}$ up
to order $l$, i.e.
\begin{equation}
\check{F}(z,w)=F_{{\rm N}}(z,w)+wO(\|(z,w)\|^{l}).\label{eq:GermFixingAxis}
\end{equation}
In particular, for $(z,0)\in L$, we have $\check{F}(z,0)=(\lambda z,0)\in L.$

Finally consider the Blow-up $\Pi:\widehat{\mathbb{C}^{2}}\to\mathbb{C}^{2}$
of $\mathbb{C}^{2}$ at the origin. Then the lift $F$ of $\check{F}$
to $\widehat{\mathbb{C}^{2}}$ leaves invariant the proper transform
$\hat{L}$ of $L$ and hence its complement $\widehat{\mathbb{C}^{2}}\backslash\hat{L}$
which is isomorphic to $\mathbb{C}^{2}$ via the coordinates $(x,y)=(z/w,w)$
defined (after extending through the exceptional divisor $E:=\Pi^{-1}((0,0))$)
on all of $\widehat{\mathbb{C}^{2}}\backslash\hat{L}$. So $F$ induces
an automorphism of $\mathbb{C}^{2}$ in these coordinates. The exceptional
divisor $E$ restricted to this $\mathbb{C}^{2}$ is $E'=\mathbb{C}\times\{0\}$.
For $(x,y)\in\mathbb{C}^{2}$, let $(x_{n},y_{n}):=F^{n}(x,y)$. Then
we have

\[
x_{1}=\frac{\lambda^{2}x(1-xy^{2}/2)+O(\|(xy,y)\|^{l})}{1-xy^{2}/2+O(\|(xy,y)\|^{l})}=\lambda^{2}x+R(x,y),
\]
where $R(x,y)=xO(\|(xy,y)\|^{l})+O(\|(xy,y)\|^{l})$ near $E'$ and
hence
\[
F(x,y)=(\lambda^{2}x+R(x,y),\overline{\lambda}y(1-xy^{2}/2)+yO(\norm{(xy,y)}^{l})).
\]
In particular $\mathcal{C}=E'\cong\mathbb{C}$ is a \emph{Siegel curve}
for $F$, i.e. $F(x,0)=(\lambda^{2}x,0)$ for all $(x,0)\in E'$.
The local basin $\check{B}$ lifts to the $F$-invariant set
\[
B=\Pi^{-1}(\check{B})=\{(x,y)\in\mathbb{C}^{2}\mid xy^{2}\in S(r_{0},\theta_{0}),|x|<\min\{|y|^{2\gamma_{0}},|y|^{\gamma_{0}-1}\}\},
\]
where $\gamma_{0}=\frac{\beta_{0}}{1-\beta_{0}}\in(0,1)$ (see Figure~\ref{fig:Argument-components}).

\section{\label{sec:The-Fatou-component}The Fatou component}

In the following we will examine the dynamics of $F$ near the invariant
curve $\mathcal{C}$ and on the lifted local basin $B$ and show that
the corresponding global basin $\Omega:=\Pi^{-1}(\check{\Omega})=\bigcup_{n\in\mathbb{N}}F^{-n}(B)$
is still a Fatou component.

For $(x,y)\in B$ and $n\in\mathbb{N}$, let $U:=1/(xy^{2})$ and
$U_{n}:=1/(x_{n}y_{n}^{2})$. Then the local basin can be written
as
\[
B=\{(x,y)\in\mathbb{C}^{2}\mid U\in H(R_{0},\theta_{0}),|x|<\min\{|y|^{2\gamma_{0}},|y|^{\gamma_{0}-1}\}\},
\]
where $R_{0}=1/(2r_{0})$ and for $R>0$ and $\theta\in(0,\pi/2)$,
the set
\[
H(R,\theta):=\{U\in\mathbb{C}\mid\Re U>R,|\arg(U)|<\theta\}
\]
is a sector ``at infinity''. Now \cite[Lemma~2.5]{BracciRaissyStensones.2017.AutomorphismsofmathbbC2withaninvariantnon-recurrentattractingFatoucomponentbiholomorphictomathbbCtimesmathbbC}
implies:
\begin{lem}
\label{lem:CharOmegaNew}For $(x,y)\in\Omega$, we have as $n\to\infty$
locally uniformly
\begin{enumerate}
\item \label{enu:UnLiken}$U_{n}\sim n$,
\item \label{enu:ynLikesqrtninverse}$|y_{n}|\approx n^{-1/2}$,
\item \label{enu:xnlikeconstant}$|x_{n}|\approx1$ (i.e.\ $x_{n}$ is
locally bounded away from $0$ and $\infty$).
\end{enumerate}
Moreover the lower bound in (\ref{enu:UnLiken}) and upper bound in
(\ref{enu:ynLikesqrtninverse}) are uniform in $B$.
\end{lem}

In \cite[Proposition~3.2]{Reppekus.2019.PeriodiccyclesofattractingFatoucomponentsoftypemathbbCtimesmathbbCd-1inautomorphismsofmathbbCd}
the author further examines the stable orbits of $\check{F}$ near
the origin and shows in particular:
\begin{prop}
\label{prop:StableOrbitsBelow}For $\check{F}$ as in Theorem~\ref{thm:BRS}
the stable set of $(0,0)$ is $W^{s}(\{(0,0)\})=\check{\Omega}\cup\{(0,0)\}$,
i.e.\ all orbits of $\check{F}$ that converge to $(0,0)$ are contained
in $\check{\Omega}$.
\end{prop}

This is enough to show:
\begin{prop}
$\Omega$ is a Fatou component and $W^{s}(\mathcal{C})=\Omega\cup\mathcal{C}$.
\end{prop}

\begin{proof}
Let $V(B)$ be the Fatou component containing $B$. Lemma~\ref{lem:CharOmegaNew},
Parts~(\ref{enu:ynLikesqrtninverse}) and (\ref{enu:xnlikeconstant})
show that the family $\{F^{n}\}_{n\in\mathbb{N}}$ is locally uniformly
bounded on $\Omega$, so by Montel's theorem, it is a normal family
on $\Omega$, hence we have $\Omega\subseteq V(B)$.

For any limit map $F_{\infty}=\lim_{k\to\infty}F^{n_{k}}$ for a subsequence
$\{n_{k}\}_{k}\subseteq\mathbb{N}$, the image $F_{\infty}(\Omega)$
is contained in the exceptional divisor $E$, so by the identity principle,
so is the image $F_{\infty}(V(B))$. In particular, for any $(x,y)\in V(B)$,
we have $(z_{n},w_{n})=(x_{n}y_{n},y_{n})\to(0,0)$, so Proposition~\ref{prop:StableOrbitsBelow}
shows $(z,w)=(xy,y)\in\check{\Omega}$ or $(x,y)\in\Omega$. Thus
we have the opposite inclusion $V(B)\subseteq\Omega$.

Let $P_{\infty}$ be the unique point in $E\backslash\mathcal{C}$.
Proposition~\ref{prop:StableOrbitsBelow} shows $W^{s}(E)=\Omega\cup E$,
hence $W^{s}(\mathcal{C})=(\Omega\cup E)\backslash W^{s}(\{P_{\infty}\})$,
but by Lemma~\ref{lem:CharOmegaNew}, Part~(\ref{enu:xnlikeconstant})
no orbit in $\Omega$ can converge to $P_{\infty}$, so $W^{s}(\{P_{\infty}\})=\{P_{\infty}\}$
and $W^{s}(\mathcal{C})=\Omega\cup\mathcal{C}$.
\end{proof}

\section{\label{sec:Cylinder-coordinates}Cylinder coordinates}

In this section we show that $F$ is conjugated on $\Omega$ to the
translation $(z,w)\mapsto(z+1,w)$. We use a Fatou coordinate and
a second local coordinate introduced in \cite{BracciRaissyZaitsev2013Dynamicsofmulti-resonantbiholomorphisms}
and \cite{BracciRaissyStensones.2017.AutomorphismsofmathbbC2withaninvariantnon-recurrentattractingFatoucomponentbiholomorphictomathbbCtimesmathbbC}
respectively to construct a global second coordinate. (By the current
reversion of this paper, the author has given a more direct construction
in \cite{Reppekus.2019.PeriodiccyclesofattractingFatoucomponentsoftypemathbbCtimesmathbbCd-1inautomorphismsofmathbbCd}.)

\cite[Sections~3 and 4]{BracciRaissyStensones.2017.AutomorphismsofmathbbC2withaninvariantnon-recurrentattractingFatoucomponentbiholomorphictomathbbCtimesmathbbC}
show, again setting $U=1/(xy^{2})$:
\begin{lem}
\label{lem:FatouCoordBRS}There exists a map $\psi:\Omega\to\mathbb{C}$
such that $\psi(x,y)=U+c\log(U)+O(U^{-1})$ as $(x,y)\to E'$ and
\[
\psi\circ F=\psi+1,
\]
 and a map $\sigma:\Omega_{0}:=\psi^{-1}(\psi(B))\to\mathbb{C}^{*}$
such that $\sigma(x,y)=y+O(U^{-\alpha})$ as $(x,y)\to E'$ with $\alpha\in(1-\beta_{0},1)\subseteq(1/2,1)$
and
\[
\sigma\circ F=\overline{\lambda}e^{-1/(2\psi)}\sigma.
\]
Furthermore
\[
(\psi,\sigma):\Omega_{0}\to\psi(B)\times\mathbb{C}^{*}
\]
is biholomorphic and $\psi(\Omega_{0})=\psi(B)$ sits between sectors
at infinity $H(\tilde{R},\tilde{\theta})\subseteq\psi(B)\subseteq H(R_{1},\theta_{1})$
for some $\tilde{R}\ge R_{1}>0$ and $0<\tilde{\theta}\le\theta_{1}<\pi/2$.
\end{lem}

\begin{rem}
\label{rem:FatouCoordLimitEquiv}In particular, by Lemma~\ref{lem:CharOmegaNew},
this implies $\psi(x_{n},y_{n})\sim U_{n}$, $\sigma(x_{n},y_{n})\sim y_{n}$,
and $\psi(x_{n},y_{n})\sim n$ and $\sigma(x_{n},y_{n})\approx n^{-1/2}$
as $n\to+\infty$ where the lower bound on $\psi(x_{n},y_{n})$ is
uniform in $\Omega_{0}$.
\end{rem}

To construct our global second coordinate, we need the following lemma
comparing the harmonic series and the logarithm:
\begin{lem}
\label{lem:HarmonicVsLog}For $\zeta\in\mathbb{C}$ such that $\Re\zeta>0$
we have
\[
\lim_{n\to\infty}\sum_{j=0}^{n-1}\frac{1}{\zeta+j}-\log\paren[{\bigg}]{\frac{\zeta+n}{\zeta}}=h(\zeta)=O\paren[{\bigg}]{\frac{1}{\zeta}}
\]
and both the limit and the bound are uniform for $\Re\zeta>R$ for
any fixed $R>0$.
\end{lem}

\begin{proof}
For $m<n$, we have
\begin{align*}
\abs[{\bigg}]{\sum_{j=m}^{n-1}\frac{1}{\zeta+j}-\log{\paren[{\bigg}]{\frac{\zeta+n}{\zeta+m}}}} & \le\sum_{j=m}^{\infty}\abs[{\bigg}]{\frac{1}{\zeta+j}-\log{\paren[{\bigg}]{\frac{\zeta+j+1}{\zeta+j}}}}\\
 & =\sum_{j=m}^{\infty}\abs[{\bigg}]{\frac{1}{\zeta+j}-\log{\paren[{\bigg}]{1+\frac{1}{\zeta+j}}}}\\
 & =\sum_{j=m}^{\infty}O(|\zeta+j|^{-2})\\
 & =O(1/|\zeta+m|)
\end{align*}
For $m\to\infty$ this shows uniform convergence and for $m=0$ and
$n\to\infty$ it follows that the limit is $O(1/\zeta)$.
\end{proof}

\begin{prop}
\label{prop:AltExtraFatouCoord}There exists a map $\tau:\Omega_{0}\to\mathbb{C}^{*}$
bijective on each fibre $\psi^{-1}(p)$ for $p\in\psi(B)$ such that
\begin{equation}
\tau\circ F=\overline{\lambda}\tau\label{eq:AltExtraFatou}
\end{equation}
and $\tau(x,y)=\sqrt{\psi(x,y)}\sigma(x,y)+\sigma(x,y)O(\psi(x,y)^{-1/2})$
as $(x,y)\to E'$.
\end{prop}

\begin{rem}
\label{rem:tauAsympEquiv}By Remark~\ref{rem:FatouCoordLimitEquiv},
we have $x_{n}=(\sqrt{U_{n}}y_{n})^{2}\sim(\tau(x_{n},y_{n}))^{2}$.
\end{rem}

\begin{proof}
Let $(x,y)\in\Omega_{0}$ and $n\in\mathbb{N}$. Note first that $\psi(x_{n},y_{n})\in H(R_{1},\theta_{1})$,
so the square root $\sqrt{\psi(x_{n},y_{n})}$ is well-defined by
choosing its values in the right half plane and we can define
\begin{align*}
\tau_{n}(x,y) & :=\lambda^{n}\sqrt{\psi(x_{n},y_{n})}\sigma(x_{n},y_{n})\\
 & =\sqrt{\psi(x,y)+n}\exp\paren[{\bigg}]{-\frac{1}{2}\sum_{j=0}^{n-1}\frac{1}{\psi(x,y)+j}}\sigma(x,y).
\end{align*}
We can consider each $\tau_{n}$ as a map $\psi(B)\times\mathbb{C}^{*}\to\mathbb{C}^{*}$
in variables $\psi$ and $\sigma$ given by
\[
\tau_{n}(\psi,\sigma)=\exp\paren[{\bigg}]{\frac{1}{2}{\paren[{\bigg}]{\log{\paren[{\bigg}]{\frac{\psi+n}{\psi}}}-\sum_{j=0}^{n-1}\frac{1}{\psi+j}}}}\sqrt{\psi}\sigma,
\]
and by Lemma~\ref{lem:HarmonicVsLog} we have
\[
\tau(\psi,\sigma):=\text{\ensuremath{\lim_{n\to\infty}}}\tau_{n}(\psi,\sigma)=\exp\paren[{\bigg}]{\frac{1}{2}h(\psi)}\sqrt{\psi}\sigma=(1+O(1/\psi))\sqrt{\psi}\sigma=\sqrt{\psi}\sigma+\sigma O(\psi^{-1/2}).
\]
$\tau$ is clearly bijective on each fibre and satisfies (\ref{eq:AltExtraFatou})
since
\begin{align*}
\tau_{n}\circ F & =\sqrt{\psi\circ F+n}\exp\paren[{\bigg}]{-\frac{1}{2}\sum_{j=0}^{n-1}\frac{1}{\psi\circ F+j}}\sigma\circ F\\
 & =\overline{\lambda}\sqrt{\psi+n+1}\exp\paren[{\bigg}]{-\frac{1}{2}\sum_{j=0}^{n}\frac{1}{\psi+j}}\sigma=\overline{\lambda}\tau_{n+1}.\qedhere
\end{align*}
\end{proof}
Now we can extend $\tau$ using the functional equation (\ref{eq:AltExtraFatou}):
\begin{prop}
\label{prop:FatouCoordGlobal}$(\psi,\tau):\Omega_{0}\to\psi(B)\times\mathbb{C}^{*}$
extends to a biholomorphism $\Phi:\Omega\to\mathbb{C}\times\mathbb{C}^{*}$
given by
\[
\Phi(x,y)=(\psi(x,y),\lambda^{n}\tau(F^{n}(x,y))
\]
for $(x,y)\in F^{-n}(B)$ and conjugating $F$ to $(z,w)\mapsto(z+1,\overline{\lambda}w)$.
\end{prop}

\begin{proof}
$(\psi,\tau)$ is injective on $\Omega_{0}$ by Proposition~\ref{prop:AltExtraFatouCoord}.
Let $p\in\Omega$. Then there exists $n\in\mathbb{N}$ such that $F^{n}(p)\in B\subseteq\Omega_{0}$.
For $m<n$ such that $F^{m}(p)$ and $F^{n}(p)$ lie in $B$, we have
\[
\lambda^{n}\tau(F^{n}(p))=\lambda^{n}\tau(F^{n-m}(F^{m}(p)))=\lambda^{m}\tau(F^{m}(p)),
\]
so $\Phi$ is well-defined. $\Phi$ is moreover injective as for any
$p,q\in\Omega$ there exists $n\in\mathbb{N}$ such that $F^{n}(p)$
and $F^{n}(q)$ lie in $B\subseteq\Omega_{0}$ where $(\psi,\tau)$
is injective.

To show surjectivity take $(\zeta,\xi)\in\mathbb{C}\times\mathbb{C}^{*}$.
Then there exists $n\in\mathbb{N}$ such that $\zeta+n\in H(\tilde{R},\tilde{\theta})\subseteq\psi(B)$
and hence $(\zeta+n,\lambda^{-n}\xi)\in\psi(B)\times\mathbb{C}^{*}=\im(\psi,\tau)$,
i.e.\ there exists $p\in\Omega_{0}$ such that $(\psi,\sigma)(p)=(\zeta+n,\lambda^{-n}\xi)$
and hence $\Phi(F^{-n}(p))=(\zeta,\xi)$.
\end{proof}
The multiplicative term $\overline{\lambda}$ in the second component
can always be eliminated, since the biholomorphic map $(z,w)\mapsto(z,\lambda^{z}w)$
conjugates $(z,w)\mapsto(z+1,\overline{\lambda}w)$ to $(z,w)\mapsto(z+1,w)$,
yielding the following corollary:
\begin{cor}
\label{cor:EliminateRotTerm}There exists a biholomorphic map $\Psi:\Omega\to\mathbb{C}\times\mathbb{C}^{*}$
conjugating $F$ to $(z,w)\mapsto(z+1,w)$.
\end{cor}

The arguments in this section rely only on the internal dynamics on
$\Omega$ described by the coordinates in Lemma~\ref{lem:FatouCoordBRS},
that have been constructed in \cite{BracciRaissyZaitsev2013Dynamicsofmulti-resonantbiholomorphisms}
and \cite{BracciRaissyStensones.2017.AutomorphismsofmathbbC2withaninvariantnon-recurrentattractingFatoucomponentbiholomorphictomathbbCtimesmathbbC}
for any automorphism of the form (\ref{eq:BZgerms}). Hence we have
moreover shown:
\begin{prop}
\label{prop:BRSniceCoords}Let $\check{F}$ and $\check{\Omega}$
be as in Theorem~\ref{thm:BRS}. Then there exists a biholomorphic
map $\check{\Psi}:\check{\Omega}\to\mathbb{C}\times\mathbb{C}^{*}$
conjugating $\check{F}$ to $(z,w)\mapsto(z+1,w)$.
\end{prop}

\section{\label{sec:Limit-set}Limit sets}

We use the coordinates from the previous section to identify the limit
sets of orbits in $\Omega$ and the images of limit functions, concluding
the proof of Theorem~\ref{thm:PuncturedSiegelCyl}.
\begin{lem}
\label{lem:LimitSetsofPts}For $(x,y)\in\Omega$, we have $\omega_{F}(x,y)=\tau(x,y)^{2}S^{1}\times\{0\}$.
\end{lem}

\begin{proof}
By Lemma~\ref{lem:CharOmegaNew}, we have $y_{n}\to0$ and by Remark~\ref{rem:tauAsympEquiv},
we have
\[
x_{n}\sim\tau(x_{n},y_{n})^{2}=\lambda^{2n}\tau(x,y)^{2}.
\]
Since $\lambda$ is an irrational rotation, $x_{n}$ accumulates on
all of $\tau(x,y)^{2}S^{1}$.
\end{proof}
\begin{cor}
\label{cor:LimitSetIsCstar}$\omega_{F}(B)=\mathbb{C}^{*}\times\{0\}$
and any limit function $F_{\infty}:\Omega\to\mathbb{C}^{*}\times\{0\}$
of a convergent subsequence of $\{F^{n}\}_{n}$ is surjective. Postcomposition
of $F_{\infty}$ with a rotation of $\mathbb{C}^{*}\times\{0\}$ yields
precisely all possible such limit functions.
\end{cor}

\begin{proof}
The map $\tau:\Omega_{0}\to\mathbb{C}^{*}$ is surjective, so $\omega(B)=\mathbb{C}^{*}\times\{0\}$.
Every limit function $F_{\infty}$ is not constant by Lemma~\ref{lem:LimitSetsofPts}
and by Picard's theorem satisfies $F_{\infty}(\Omega)=\mathbb{C}^{*}\times\{0\}$.
\end{proof}
This concludes the proof of Theorem~\ref{thm:PuncturedSiegelCyl}.


\begin{thebibliography}{BTBP20}
\bibitem[ABFP19]{ArosioBeniniFornessPeters.2019.DynamicsoftranscendentalHenonmaps.}
L. Arosio, A.M. Benini, J.~E. Forn{\ae}ss, and H. Peters, \emph{Dynamics
of transcendental {H}{\'e}non maps}, Math. Ann. \textbf{373} (2019),
no.~1-2, 853--894.

\bibitem[BS91]{BedfordSmillie.1991.Polynomialdiffeomorphismsof(mathbbC2).II:Stablemanifoldsandrecurrence.}
E. Bedford and J. Smillie, \emph{Polynomial diffeomorphisms of {$\mathbb{C}^{2}$}.
{II}: Stable manifolds and recurrence}, J. Amer. Math. Soc. \textbf{4}
(1991), no.~4, 657--679.

\bibitem[BTBP20]{BocBracciPeters.2019.AutomorphismsofmathbbC2withnon-recurrentSiegelcylinders}
L. Boc~Thaler, F. Bracci, and H. Peters, \emph{Automorphisms of {$\mathbb{C}^{2}$}
with parabolic cylinders}, J. Geom. Anal. (2020), \href{https://doi.org/10.1007/s12220-020-00403-4}{doi:10.1007/s12220-020-00403-4}.

\bibitem[BZ13]{BracciZaitsev.2013.Dynamicsofone-resonantbiholomorphisms.}
F. Bracci and D. Zaitsev, \emph{Dynamics of one-resonant biholomorphisms},
J. Eur. Math. Soc. \textbf{15} (2013), no.~1, 179--200.

\bibitem[BRZ13]{BracciRaissyZaitsev2013Dynamicsofmulti-resonantbiholomorphisms}F.
Bracci, J. Raissy, and D. Zaitsev, \emph{Dynamics of multi-resonant
biholomorphisms}, Int. Math. Res. Not. IMRN \textbf{2013} (2013),
no.~20, 4772--4797.

\bibitem[BRS]{BracciRaissyStensones.2017.AutomorphismsofmathbbC2withaninvariantnon-recurrentattractingFatoucomponentbiholomorphictomathbbCtimesmathbbC}
F. Bracci, J. Raissy, and B. Stens{\o}nes, \emph{Automorphisms
of {$\mathbb{C}^{k}$} with an invariant non-recurrent attracting
{F}atou component biholomorphic to {$\mathbb{C}\times(\mathbb{C}^{*})^{k-1}$}},
to appear in J. Eur. Math. Soc.

\bibitem[FS95]{FornessSibony.1995.Classificationofrecurrentdomainsforsomeholomorphicmaps.}
J.E. Forn{\ae}ss and N. Sibony, \emph{Classification of recurrent
domains for some holomorphic maps}, Math. Ann. \textbf{301} (1995),
no.~4, 813--820.

\bibitem[JL04]{JupiterLilov.2004.InvariantnonrecurrentFatoucomponentsofautomorphismsof(mathbbC2).}
D. Jupiter and K. Lilov, \emph{Invariant nonrecurrent {F}atou components
of automorphisms of {$\mathbb{C}^{2}$}}, Far East J. Dyn. Syst.
\textbf{6} (2004), no.~1, 49--65.

\bibitem[LP14]{LyubichPeters.2014.ClassificationofinvariantFatoucomponentsfordissipativeHenonmaps}
M. Lyubich and H. Peters, \emph{Classification of invariant {F}atou
components for dissipative {H}{\'e}non maps}, Geom. Funct. Anal.
\textbf{24} (2014), no.~3, 887--915.

\bibitem[PVW08]{PetersVivasWold.2008.AttractingbasinsofvolumepreservingautomorphismsofBbbCk}
H. Peters, L.R. Vivas, and E.F. Wold, \emph{Attracting basins of volume
preserving automorphisms of {$\mathbb{C}^{k}$}}, Internat. J. Math.
\textbf{19} (2008), no.~7, 801--810.

\bibitem[Rep20]{Reppekus.2019.PeriodiccyclesofattractingFatoucomponentsoftypemathbbCtimesmathbbCd-1inautomorphismsofmathbbCd}
J. Reppekus, \emph{Periodic cycles of attracting {F}atou components
of type {$\mathbb{C}\times(\mathbb{C}^{*})^{d-1}$} in automorphisms
of {$\mathbb{C}^{d}$}}, arXiv e-prints (2020), \href{https://arxiv.org/abs/1905.13152v2}{arXiv:1905.13152v2}.

\bibitem[RR88]{RosayRudin.1988.HolomorphicmapsfrombfCntobfCn} J.-P.
Rosay and W. Rudin, \emph{Holomorphic maps from {${\bf C}^{n}$}
to {${\bf C}^{n}$}}, Trans. Amer. Math. Soc. \textbf{310} (1988),
no.~1, 47--86.

\bibitem[SV14]{StensonesVivas2014BasinsofAttractionofAutomorphismsinC3}B.
Stens{\o}nes and L. Vivas, \emph{Basins of attraction of automorphisms
in {$\mathbb{C}^{3}$}}, Ergodic Theory Dynam. Systems \textbf{34}
(2014), 689--692.

\bibitem[Ued86]{Ueda.1986.Localstructureofanalytictransformationsoftwocomplexvariables.I}
T. Ueda, \emph{Local structure of analytic transformations of two
complex variables {I}}, J. Math. Kyoto Univ. \textbf{26} (1986),
no.~2, 233--261.

\bibitem[Ued08]{Ueda.2008.HolomorphicmapsonprojectivespacesandcontinuationsofFatoumaps.}
T. Ueda, \emph{Holomorphic maps on projective spaces and continuations
of Fatou maps}, Michigan Math. J. \textbf{56} (2008), no.~1, 145--153.

\bibitem[Var00]{Varolin.2000.Thedensitypropertyforcomplexmanifoldsandgeometricstructures.II.}
D. Varolin, \emph{The density property for complex manifolds and geometric
structures. {II}}, Internat. J. Math. \textbf{11} (2000), no.~6,
837--847.

\bibitem[Var01]{Varolin.2001.Thedensitypropertyforcomplexmanifoldsandgeometricstructures.}
D. Varolin, \emph{The density property for complex manifolds and geometric
structures}, J. Geom. Anal. \textbf{11} (2001), no.~1, 135--160.
\end{thebibliography}
\end{document}